\documentclass[12pt,reqno]{amsart}
\usepackage{amsmath, amssymb, amsfonts, xcolor}
\usepackage{hyperref}
\numberwithin{equation}{section}
\usepackage{amsthm}
\newtheorem{thm}{Theorem}[section]

\newtheorem{prop}[thm]{Proposition}
\newtheorem{lem}[thm]{Lemma}
\newtheorem{conj}[thm]{Conjecture}
\newtheorem{dfn}[thm]{Definition}

\newtheorem{remark}[thm]{Remark}
\newtheorem{cor}[thm]{Corollary}

\usepackage[a4paper, top = 1.2in, bottom = 1.2in, left = 1.2in, right = 1.2in]{geometry}

\numberwithin{equation}{section}

\newcommand{\N}{\mathbb{N}}
\newcommand{\Q}{\mathbb{Q}}

\newcommand{\R}{\mathbb{R}}

\newcommand{\Z}{\mathbb{Z}}

\newcommand{\mcO}{\mathcal{O}}

\newcommand{\mcL}{\mathcal{L}}

\newcommand{\mfa}{\mathfrak{a}}

\newcommand{\mfn}{\mathfrak{n}}

\newcommand{\Cl}{\mathrm{Cl}}

\newcommand{\Gal}{\mathrm{Gal}}

\newcommand{\tors}{\mathrm{tors}}

\newcommand{\ra}{\rightarrow}

\def\1{1\!\!1}
\newcommand{\mrm}[1]{\mathrm{#1}}

\title[Lower bounds-abelian varieties-Lehmer conjecture]{
Lehmer-type bounds and counting rational points of bounded heights on Abelian varieties}

\author[N. Kumar]{Narasimha Kumar}
\address[N. Kumar]{Department of Mathematics, Indian Institute of Technology Hyderabad, Kandi, Sangareddy 502285, INDIA.}
\email{narasimha@math.iith.ac.in}

\author[S. Sahoo]{Satyabrat Sahoo}
\address[S. Sahoo]{Department of Mathematics, Indian Institute of Technology Hyderabad, Kandi, Sangareddy 502285, INDIA.}
\email{ma18resch11004@iith.ac.in}

\date{\today}

\begin{document}
	\begin{abstract}
	In this article, we study Lehmer-type bounds for the N\'eron-Tate height of $\bar{K}$-points on abelian varieties $A$ over number fields $K$. Then, we estimate
	the number of $K$-rational points on $A$ with N\'eron-Tate height $\leq \log B$ for $B\gg 0$. This estimate involves a constant $C$, which is not explicit. However, for elliptic curves and the product of elliptic curves over $K$, we make the constant explicitly computable.
	\end{abstract}		
	\maketitle
\section{Introduction}
Throughout this article, $K$ denotes a number field, and $\bar{K}$ denotes its algebraic closure. Let $\mathbb{P}$ denote the set of all rational primes. Let $E/K$ (resp., $A/K$) denote an elliptic curve (resp., abelian variety) over $K$. There has been extensive literature on Lehmer-type bounds for the N\'eron-Tate height on $A(\bar{K})$, and counting the number of rational points of bounded heights on $E/K$ (cf.
\cite{M84}, \cite{M89}, \cite{GM17}, \cite{BZ04},\cite{N21}).
%Let $\mathbb{P}$ denote the set of all rational primes.

\subsection{Lehmer-type bounds}
In the literature, the following conjecture is known as the Lehmer conjecture.
Let $\mcL$ be a symmetric ample line bundle on $A$, and let $\hat{h}$ be the N\'eron-Tate height on $A(\bar{K})$ associated to $L$.
\begin{conj}[\cite{M84}, \cite{DH00}]
	\label{Lehmer conjecture}
For any abelian variety $A/K$ of dimension $g$, there exists a constant $C=C(A,K)\in \R^+$ (depending on $A,K$) such that
%$\hat{h}(P) \geq CD^{-\frac{1}{g}}$, for any point $ P \in A(\bar{K}) \setminus A(\bar{K})_{\tors}$ of degree $D=[K(P): K]$.
	$\hat{h}(P) \geq \frac{C}{D^{\frac{1}{g}}}$ for any non-torsion point $ P \in A(\bar{K})$ of degree $D=[K(P): K]$.
\end{conj}
In the celebrated work of Amoroso and Masser, in~\cite{AM16}, they studied Lehmer-type bounds for any $\alpha \in \bar{\Q}^\times$, which is not a root of unity, such that $\Q(\alpha)$ is Galois over $\Q$.
For  elliptic curves $E$ over $K$, in~\cite{AM80}, Anderson and Masser proved that $\hat{h}(P) \geq \frac{C}{D^{10} (\log{D})^{6}}$ for any non-torsion point $P \in E(\bar{K})$ of degree $D$.
	In \cite{S81}, Silverman sharpened this bound to $\frac{C}{D^{2}}$ when $K$ is abelian.
	% Later in \cite{L83}, Laurent improved the bound \eqref{AM_Lehmer bound CM} to $\frac{C (\log\log{D})^3}{D (\log{D})^{3}}$.
	In~\cite{M89}, Masser improved the bound of ~\cite{AM80} to $\frac{C}{D (\log{D})^{2}}$ when $K$ is abelian. In \cite{GM17}, Galateau and Mah\'e proved  Conjecture~\ref{Lehmer conjecture} for all non-torsion points $P \in E(\bar{K})$ 
	with $K(P)$ Galois over $K$. 
	
	For $A/K$ with dimension $g$, in \cite{M84}, Masser proved that $\hat{h}(P) \geq \frac{C}{ D^{2g+6+\frac{2}{g}}}$ for any non-torsion point $P \in A(\bar{K})$ of degree $D$. In \cite{M86}, Masser improved this bound to $ \frac{C}{D^{2g+1}( \log{D})^{2g}}$. Recently, in \cite{GR22}, Gaudron and R\'emond gave a completely explicit version of \cite{M86} and determined the constant $C$ explicitly in terms of $A, g$.
    In~\cite[Theorem 1.9]{GM17}, Galateau and Mah\'e improved the bound of \cite{M86} to $\frac{C}{D (D\log{D})^{\frac{2}{g}}}$, if $A$ is a power of an elliptic curve. In~\cite{DH00}, David and Hindry studied Lehmer-type bounds for CM abelian varieties $A$ of dimension $g$.
    % and prove that $\hat{h}(P) \geq C(A, \mcL) \times D^{-\frac{1}{g}} \times \left(\frac{\log \log (3D)}{\log (3D)} \right)^{\kappa(g)}$ for any non-torsion point $P \in A(\bar{K})$ of degree $D$, where $\kappa(g)=(2g . (g+1)!)^{g+2}$.
	
\begin{itemize}
	\item  	
	In the first part of this article, we study Lehmer-type bounds for  $\hat{h}$ on $A/K$ of dimension $g$. In Theorem~\ref{abelian mainthm3}, we prove that 
%	$\hat{h}(P) \geq \frac{C}{D^{2g+1}}$ for all non-torsion Galois points $P \in A(\bar{K})$ of degree $D \geq 1$. In Theorem ~\ref{abelian mainthm3}, we prove 
$\hat{h}(P) \geq \frac{C}{(D\log{D})^{2g}}$ for all non-torsion points $P \in A(\bar{K})$ of degree $D \geq 3$ with $K(P)$ Galois over $K$. In Corollary~\ref{abelian mainthm2}, we show that for any $\epsilon >0$, the inequality $\hat{h}(P) \geq \frac{C}{D^{2g+\epsilon}}$ holds for all non-torsion points $P \in A(\bar{K})$ of degree $D \geq 1$ with $K(P)$ Galois over $K$. Theorem~\ref{abelian mainthm3} and Corollary~\ref{abelian mainthm2} improve the bounds in ~\cite{M86} and ~\cite{M84}, respectively.
%	Also, generalize~\cite[Theorem 1.10]{GM17} to abelian varieties $A/K$.
\end{itemize}

\subsection{Counting the rational points on $A/K$}
%Let $\hat{h}$ be the  N\'eron-Tate height on $A(K)$.
%Let $H_K$ be the multiplicative height on the projective space $\mathbb{P}^n(K)$ 
%In \cite[Chap VIII, Theorem 5.11]{S86}, Silverman proved that, for any $B \in \R^+$, the set $N_{\mathbb{P}^n,K}(B) := \{P \in \mathbb{P}^n(K)| H_K(P) \leq B\}$ is finite. 
Let $\mcL$ be a symmetric ample line bundle on $A/K$. Let $H_K$, $\hat{h}_K$ be the multiplicative, N\'eron-Tate heights on $A(K)$ associated to $\mcL$. For any $B \in \R^+,$ define $N_{A,K}(B) := \# \{P \in A(K) : H_K(P) \leq B\}$ and $\hat{N}_{A,K}(B) := \# \{P \in A(K): \hat{h}_K(P) \leq \log B\}$. By~\cite[Theorem B.2.3]{HS00},  $N_{A,K}(B)$ is finite. For elliptic curve $E/K$, there is a vast amount of literature on estimating $N_{E,K}(B)$
and $\hat{N}_{E,K}(B)$ for $B \gg 0$.

In \cite{BZ04}, Bombieri and Zannier proved that, for any $E/ \Q$ with $\#E(\Q)[2]=4$, there exist constants
$C, C^\prime \in \R^+$ such that $\hat{N}_{E, \Q}(B) \leq B^{ \left(\frac{C}{\sqrt{\log \log B}} \right)}$ and  
 $N_{E, \Q}(B) \leq B^{ \left(\frac{C^\prime}{\sqrt{\log \log B}} \right)}$ for all $B \gg 0$. These proofs used a lower bound of $\hat{h}_K$ on $E(K) \setminus E_{\tors}(K)$ (cf. \cite[Theorem 0.3]{HS88}). In~\cite{N21}, Naccarato improved the exponents to $\frac{*}{\log \log B}$ for all $E/ \Q$ with $E(\Q)[2] \neq 0$ and these proofs used an improved lower bound of $\hat{h}_K$ on $E(K) \setminus E_{\tors}(K)$ (cf. \cite[Theorem 2]{P06}).

\begin{itemize}
 \item In the second part of this article, we prove a variant of the result in~\cite{N21} for abelian varieties $A/K$ with $A(K)[p]\neq 0$ for some $p\in \mathbb{P}$ (cf. Theorem~\ref{main result2 abelian variety for  over K}). In the proof, we need a covering lemma (cf. Lemma~\ref{covering lemma}),  a result of Gaudron and R\'emond for the lower bound of $\hat{h}_K$ on $ A(K) \setminus A_\tors(K)$ (cf. Theorem~\ref{BS lower bound on the canonical height on A(K)}), and a result of Ooe and Top for the upper bound of the minimal number of generators of $A(K)/mA(K)$ (cf. \cite[Theorem 2]{OT89}).

Next, we show that the constants in Theorem~\ref{main result2 abelian variety for  over K} are computable  for elliptic curves over $K$ (cf. Theorem~\ref{main result2 for elliptic curve over K})
and we extend it to the product of elliptic curves over $K$ (cf. Proposition~\ref{main result1 for product of elliptic curves over K}). Here, instead of using the result of Gaudron and R\'emond, we could use a result of Petsche for an explicit lower bound of $\hat{h}_K$ on $ E(K) \setminus E_\tors(K)$ in terms of Szpiro ratio of $E$(cf. \cite[Theorem 2]{P06}), which makes the constants in Theorem~\ref{main result2 for elliptic curve over K} explicitly computable. Finally, in Proposition~\ref{prop for lower bound for product of elliptic curves}, we gave an explicit lower bound for $\hat{h}_K$ on the product of elliptic curves over $K$ in terms of Szpiro ratios of these elliptic curves. To the best of the author's knowledge, such explicit lower bounds are not known for $A$ over $K$(cf. ~\cite[Theorem 0.1]{BS04}).
\end{itemize}

\section{Lehmer-type bounds for abelian varieties $A$ over $K$}
\label{section for Lehmer Conjecture for abelian varieties}
 Let $A$ be an abelian variety defined over $K$, denoted by $A/K$, with a symmetric ample line bundle $\mcL$. Let $\hat{h} : A(\bar{K}) \rightarrow [0, \infty)$ be the N\'eron-Tate height on $A(\bar{K})$ associated to $\mcL$. Then $\hat{h}$ is a positive definite quadratic form on $A(\bar{K})/ A_{\tors}(\bar{K})$ (cf. \cite[Theorem B.5.3]{HS00}), where $A_\tors(\bar{K})$ denote torsion subgroup of $A(\bar{K})$. In this section, we study Lehmer-type bounds of $\hat{h}$ for all non-torsion points $P \in A(\bar{K})$ with $K(P)$ Galois over $K$. 
First, we recall a result of Masser in~\cite{M86}.
\begin{thm}
	\label{Mas bound for points of small height}
Let $A$, $K$ be as in Conjecture~\ref{Lehmer conjecture}. Then there exists a constant $c=c(A,K) \in \R^+$ (depending on $A,K$) such that for any extension $L$ of $K$ with degree at most $D$, we have
$\# \{P \in A(L): \hat{h}(P) \leq \frac{1}{cD}\} \leq cD^g (\log{D})^{g}$.
\end{thm}
In \cite[Th\'eor\'eme 1.1]{GR22}, Gaudron and R\'emond determined the constant $c$ of Theorem~\ref{Mas bound for points of small height} explicitly in terms of $A$ and $g$. Now, we state the main result of this section.
\begin{thm}
    	\label{abelian mainthm3}
    	Let $A/K$ be an abelian variety of dimension $g$.
    	Then there exists a constant $C=C(A, K) \in \R^+$ (depending on $A,K$) such that
    	$\hat{h}(P) \geq \frac{C}{(D \log{D})^{2g}},$ for all non-torsion points
    	$P \in A(\bar{K})$ of degree $D \geq 3$ with $K(P)$ Galois over $K$.
    \end{thm}

	\begin{proof}[Proof of Theorem~\ref{abelian mainthm3}]
Let $c \in \R^+$ be as in Theorem~\ref{Mas bound for points of small height}.
We shall prove the theorem by contradiction. Suppose there exists a non-torsion point $P \in A(\bar{K})$ with $L := K(P)$ Galois over $K$ of degree $D$ and
	\begin{equation}
		\label{contr assumption for mainthm3}
		\hat{h}(P) < \frac{C}{(D\log{D})^{2g}},
		\end{equation}
 where $C:=\frac{1}{c(c+1)^2}$. Let $H:=A_{\tors}(L)$ and $R$ be the maximal set of conjugates of $P$ over $K$ that are mutually incongruent modulo $H$. Let $t:= \#H, r:= \#R$. 
 
 We now show that $D=[L:K] \leq rt$.
 Let $G:=\Gal(L/K)= \{ \sigma_1, \dots \sigma_D \}$, where $\sigma_1, \dots \sigma_D$ are all the $K$-embeddings of $L$ in $\bar{K}$. Define a relation $\sim$ on $G$ by $\sigma_i \sim \sigma_j$ if $P^{\sigma_i}-P^{\sigma_j} \in H$. Clearly, $\sim$ is an equivalence relation on $G$. Then $G= \bigsqcup [\sigma_i]$ is a disjoint union of distinct equivalence classes $[\sigma_i]$. By definition of $H$, $\# [\sigma_i] \leq \# H=t$ for all $i$. By the definition of $R$, the number of distinct equivalence classes $\leq \# R=r$. Hence, $\# G=D \leq rt$.

 Let $M=1 +[cD^{g-1}(\log{D})^{g}]$. Since $D \geq 3$, we have
    \begin{equation}
    	\label{bound on M}
    	cD^{g-1}(\log{D})^{g} < M \leq (c+1)D^{g-1}(\log{D})^{g}.
    \end{equation}
    Let $S:= \{mQ+T: Q \in R,\ T \in H, \text{ and } m \in \Z \text{ with } 1 \leq m \leq M\}$. We will now show that all the elements of $S$ are distinct, which gives us the cardinality of $S$.

    Suppose $m_1Q_1+T_1= m_2Q_2+T_2$, for some $Q_1, Q_2 \in R$, $T_1, T_2 \in H$, and $m_1, m_2 \in \Z$ with $1 \leq m_1, m_2 \leq M$. By definition of $H$, we get  $tm_1Q_1=t m_2Q_2$. Since $\hat{h}$ is Galois invariant and $\hat{h}(P)>0$, we get $m_1= m_2$. This implies that $m_1(Q_1-Q_2)= T_2-T_1$, hence $Q_1-Q_2 \in H$. Since $Q_1, Q_2 \in R$, $Q_1=Q_2$ and hence $ T_1=T_2$. Therefore, $\#S= Mrt$. By \eqref{bound on M} and since $rt \geq D$, we get
    \begin{equation}
    	\label{lower bound on S}
   	 \#S > cD^g(\log{D})^{g}.
    \end{equation}
     We shall now find an upper bound for $S$.
Let $mQ+T \in S$. By applying the triangle inequality of the norm associated with the N\'eron-Tate height on $A(L)$ (cf.~\cite[Theorem B.5.1]{HS00}), we have
$\left( \hat{h}(mQ+T) \right)^\frac{1}{2} \leq \left( \hat{h}(mQ) \right)^\frac{1}{2}+\left( \hat{h}(T) \right)^\frac{1}{2} = m \hat{h}(P)^\frac{1}{2}$. This implies
$$\hat{h}(mQ+T) < \frac{M^2C}{(D\log{D})^{2g}} < \frac{C(c+1)^2 D^{2g-2}(\log{D})^{2g} }{(D\log{D})^{2g}}=\frac{1}{cD^2}< \frac{1}{cD}.$$
The first (resp., second) inequality follows from~\eqref{contr assumption for mainthm3} (resp., \eqref{bound on M}). Note that $m \leq M$, $C=\frac{1}{c(c+1)^2}$. So, $S \subseteq \left\{ P \in A(L): \hat{h}(P) \leq \frac{1}{cD}\right\}.$
By Theorem~\ref{Mas bound for points of small height}, we have
\begin{equation*}
	\label{upper bound for S}
	\#S \leq  c D^g (\log{D})^{g}.
\end{equation*}
This is a contradiction to the lower bound for $S$ in \eqref{lower bound on S}, hence we are done.
%Finally by \eqref{lower bound for S} and \eqref{upper bound for S}, we have
%$c D (\log{D})^{g} < |S| \leq c D^g (\log{D})^{g}$. {\color{red}This gives $D < D^g$, which is a contradiction. This completes the proof of theorem.}
\end{proof}
%The proof of Theorem~\ref{abelian mainthm3} is similar to that of \cite[Corollary 4]{M89}.

Now, we recall a definition.
\begin{dfn}
	\label{def for bounded degree}
	%Let $A/K$ be an abelian variety.
	We say a point $P \in A(\bar{K})$ is of bounded degree $f \in \N$, if there exists a finite extension $L/K$ of degree at most $f$ such that $P \in A(L)$.
\end{dfn}
In the next lemma,  we give an absolute lower bound of $\hat{h}(P)$ for $ P \in A(\bar{K}) \setminus A_\tors(\bar{K})$ is of bounded degree. More precisely;
\begin{lem}
	\label{abelian mainthm1}
	Let $A/K$ be an abelian variety, and let $f\geq 1$ be an integer. Then there exists a constant $C=C(A, K, f) \in \R^+$ (depending on $A,K,f$) such that $\hat{h}(P) \geq C$ for any $ P \in A(\bar{K}) \setminus A_\tors(\bar{K})$ of bounded degree $f$.
	%	$\hat{h}(P) \geq \frac{C}{D^g (\log{D})^{2g}}$.	
\end{lem}
The proof of the lemma above is straightforward, and it does not necessarily need Theorem~\ref{Mas bound for points of small height}. However, we give a proof using Theorem~\ref{Mas bound for points of small height}.
\begin{proof}
	Let $c \in \R^+$ be as in Theorem~\ref{Mas bound for points of small height}.
	Since Theorem~\ref{Mas bound for points of small height} is true
	for all constants $ \geq c$, we assume that $c \geq (17f)^{\frac{2}{3}}$. We shall show that
	$\hat{h}(P) \geq C$, where $C:= \frac{1}{c^4}$, for any non-torsion points $P \in A(\bar{K})$ of bounded degree $f$. We prove this statement by contradiction. Suppose there exists a non-torsion point $P \in A(\bar{K})$
	of bounded degree $f$ with
	\begin{equation}
		\label{contr assumption1}
		\hat{h}(P) < \frac{1}{c^4}.
	\end{equation}
	Take $L :=K(P)$. Since $P$ is of bounded degree $f$, $D:=[L:K] \leq f$.
	Let
	\begin{equation}
		\label{set S(p)}
		S(P) := \left\{pP : p \in \mathbb{P},  1 < p \leq m:=c^\frac{3}{2} D^{-\frac{1}{2}} \right\}.
	\end{equation}
	Since $\hat{h}$ is a quadratic form, $\hat{h}(pP)= p^2\hat{h}(P)$ for $p \in \mathbb{P}$.
	By~\eqref{contr assumption1} and~\eqref{set S(p)}, we get
	$\hat{h}(pP)  <  \frac{p^2}{c^4 } \leq \frac{c^3D^{-1}} {c^4}= \frac{1}{cD}.$
	Hence, $S(P) \subseteq \left\{ Q \in A(L): \hat{h}(Q) \leq \frac{1}{cD}\right\}.$ By Theorem~\ref{Mas bound for points of small height}, we have
	\begin{equation}
		\label{upper bound for S(p)}
		\#S(P) \leq cD^g (\log{D})^{g}.
	\end{equation}
	We now calculate a lower bound of $S(P)$ to get a contradiction. Since $P \in A(\bar{K})$ is a non-torsion point, we get $pP \neq p' P$ for primes $p \neq p'$ with $1 < p,\ p' \leq m$. This implies that $\#S(P)= \pi (m)$.
	Since $c \geq (17f)^{\frac{2}{3}}$, we get $c^\frac{3}{2} \geq 17f \geq  17D^\frac{g-1}{2g}$, and hence $m=c^\frac{3}{2} D^{-\frac{g-1}{2g}} \geq 17$. By~\cite[Corollary 1]{RS62}, we get $\pi (m) \geq \frac{m}{\log m} $.
	%		we get the number of distint primes $p \leq m$ i.e., $\pi (m)
	%		 Note that here $m \geq 17$.
	So $\#S(P) \geq  \frac{m}{\log m} $.
	%	    Again as $m \geq 17> e^2$, $\log m \geq
	%		So	$\#S(p) \geq  D (\pi (m)- \frac{\log m}{\log 2}) \geq  D \frac{\pi (m)}{2} \geq D(\frac{m}{2\log m})$.
	Since $m=c^\frac{3}{2} D^{-\frac{g-1}{2g}}$, we have
	\begin{equation}
		\label{lower bound for S(p)}
		\#S(P) \geq \frac{c^\frac{3}{2}D^{-\frac{1}{2}} }{\frac{3}{2}\log c -{\frac{1}{2}}\log D } \geq \frac{c^\frac{3}{2}D^{-\frac{1}{2}} }{\frac{3}{2}\log c }.
	\end{equation}
	By~\eqref{upper bound for S(p)} and \eqref{lower bound for S(p)}, we have
	\begin{equation}
		\label{contradiction-mainthm2}
		\frac{2 \sqrt{c}}{3 \log c}  \leq D^{g+\frac{1}{2}} (\log{D})^{g} \leq f^{g+\frac{1}{2}}  (\log f)^{g}.
	\end{equation}
	Since $\frac{\sqrt{c}}{\log c}$ is an increasing function of $c$ for $c\geq e^2
	$, we could choose $c \gg 0$ at the beginning of the proof to violate \eqref{contradiction-mainthm2}. This completes the proof of the lemma.
	%
	%{	\color{red}
		%	\begin{equation}
			%		\label{inequality for c}
			%		\frac{2c^\frac{1}{2}}{3 \log c} \leq d^{g+\frac{g-1}{2g}}  (\log d)^{g}.
			%	\end{equation}
		%By hypothesis on $c$, we have $\log c >1$ and $ c^\frac{1}{2}>(\frac{3d^{g+\frac{g-1}{2g}}
			%	(\log d)^g}{2}) $. This gives $\frac{2c^\frac{1}{2}}{3 \log c} > d^{g+\frac{g-1}{2g}}  (\log d)^{g}$,
		%which contradicts \eqref{inequality for c}.
		%	 Hence the proof of the theorem follows.}
	%	This gives $\frac{2}{3c^\frac{1}{2}}=\frac{2c^\frac{1}{2}}{3c} < d^{g+\frac{g-1}{2g}}  (\log d)^{g}$. Hence,
	%	As $c \geq (17)^{\frac{2}{3}} >e$, we get $3 \log c >3$. Hence $c^\frac{1}{2} < \frac{d^{g+\frac{g-1}{2g}} (\log d)^g}{6}$, which is a contradiction.
	%	This completes the proof of the theorem.
\end{proof}
 
\begin{remark}
\label{RemarkBogomolov}
Note that, Lemma~\ref{abelian mainthm1} can be reinterpreted as follows.	A set $X$ is said to have Bogomolov property if there exists a constant $C$ (depending on $X$) such that $\hat{h}(P) \geq C$ for all non-torsion points $ P \in A(X)$. For any abelian variety $A/K$, $f \geq 1$ be an integer, let $A_{K,f}:=\{ P\in A(\bar{K}) : [K(P) : K]\leq f \}$. By Lemma~\ref{abelian mainthm1}, the set $A_{K,f}$ has Bogomolov property.
\end{remark}

The following result is an application of Theorem~\ref{abelian mainthm3} and Lemma~\ref{abelian mainthm1}.

%\begin{cor}
%	\label{abelian mainthm2}
%	Let $A/K$ be an abelian variety of dimension $g$. Then there exists a constant $C=C(A, K) \in \R^+$ (depending on $A,K$) such that $\hat{h}(P) \geq \frac{C}{D^{2g+1}},$
%	for all non-torsion points $P \in A(\bar{K})$ of degree $D \geq1$ with $K(P)$ is Galois over $K$.
%\end{cor}
 
\begin{cor}
	\label{abelian mainthm2}
	Let $A/K$ be an abelian variety of dimension $g$ and $\epsilon \in \R^+$ be any positive real number. Then there exists a constant $C=C(A, K) \in \R^+$ (depending on $A,K$) such that $\hat{h}(P) \geq \frac{C}{D^{2g+\epsilon}}$
	for all non-torsion points $P \in A(\bar{K})$ of degree $D \geq1$ with $K(P)$ Galois over $K$.
\end{cor} 
%small improvements on [Mas86], and even then,  only in the “galois case”, i.
%
%    The exponents of $D$ in Corollary~\ref{abelian mainthm2} improve the earlier exponents in ~\cite{M84},~\cite{M86} for non-torsion points $P \in A(\bar{K})$ with $K(P)$ is Galois over $K$.
\begin{proof}
%	Since $ \frac{1}{(D \log{D})^{2g}} \geq  \frac{1}{(D)^{2g+1}}$, for $D \gg 0$ (depending on $g$).
%	By Theorem~\ref{abelian mainthm3}, the proof of the corollary follows for all $D \geq M \gg 0$ with some $M$ (depending on $g$).
%	Since $M$ depends on $c, c^\prime, c^{\prime \prime}$, $M$ depends only on $A,K$.

For any $\epsilon \in \R^+$, we get $D^\epsilon \geq (\log{D})^{2g}$  for $D \geq M \gg 0$,
where $M \in \R^+$ depends on $g$. This implies that $\frac{1}{(D \log{D})^{2g}} \geq  \frac{1}{(D)^{2g+\epsilon}}$, for $D \geq M$. By Theorem~\ref{abelian mainthm3}, the proof of the corollary follows for all $D \geq M$.

	Now, we will prove the corollary for the remaining values of $D$ with $1 \leq D \leq M$. Let $P \in A(\bar{K})$ be a non-torsion point of degree $D \leq M$, i.e., $P$ of bounded degree $M$. By Lemma~\ref{abelian mainthm1}, there exists a constant $C$ (depending on $A$, $K$, $M$) such that $\hat{h}(P) \geq C > \frac{C}{D^{2g+\epsilon}}$. Since $M$ depends on $g=\mrm{dim}(A)$, $C$ depends only on $A$ and $K$. Hence, the proof of the corollary follows.
\end{proof}
Note that the exponent of $D$ in the lower bounds for $\hat{h}$ in Theorem~\ref{abelian mainthm3} (resp., Corollary~\ref{abelian mainthm2}) improves the known exponent $2g+1$ in ~\cite{M86} (resp., $2g+6+\frac{2}{g}$ in~\cite{M84}) for all non-torsion points $P \in A(\bar{K})$ with $K(P)$ Galois over $K$.

\section{Counting rational points on abelian varieties $A$ over $K$}
\label{section for height function on abelian variety over K}
\label{section for height function on elliptic curves over K}
\label{section for canonical height}
Throughout this section, $K$, $\mcO_K$, $\mfn$, and $\Cl_K$, denote a number field, the ring of integers of $K$, an ideal of $\mcO_K$, and the class group of $K$, respectively. Let $\mathbb{P}$ denote the set of all rational primes.
%	 and $h_0(E):= \log H_0(E)$.
%\subsection{Height functions:}
%	We now recall the logarithmic, N\'eron-Tate height functions on $E(K)$ from \cite{S86}.
%	Let $M_K$ denote the set of all absolute values on $K$. For any $P \in E(K)\setminus \{ O\}$, the multiplicative height, logarithmic height of $P$ are defined by
%	$$ H_K(P):=\prod_{v \in M_K} \max\{1, |x(P)|_v \}^{[K_v : \Q_v]},\ h_K(P):=\log H_K(P).$$ The N\'eron-Tate height is defined by $\hat{h}_K(P):= \lim_{n \rightarrow \infty} \frac{1}{4^n}h_K([2^n]P).$ For any $B \in \R^+$, let $N_{E,K}(B):= \#\{P \in  E(K): H_K(P) \leq B\}$ and $\hat{N}_{E,K}(B):= \#\{P \in  E(K): \hat{h}_K(P) \leq \log B\}$.
	
	Let $A$ (resp., $E$) be an abelian variety (resp., elliptic curve) over $K$, denoted by $A/K$ (resp., $E/K$). Let $\mcL$ be a symmetric ample line bundle on $A$.
	Let $H_K$, $h_K$, $\hat{h}_K$ be the multiplicative, logarithmic, N\'eron-Tate height on $A(K)$ associated to $\mcL$ (cf. \cite[Part B]{HS00}).
	For any $B \in \R^+$, let $N_{A,K}(B):= \#\{P \in  A(K): H_K(P) \leq B\}$ and $\hat{N}_{A,K}(B):= \#\{P \in  A(K): \hat{h}_K(P) \leq \log B\}$. 
%	For $A/K$ with conductor $\mfn$, let $H_0(A):= |\mrm{Norm}_{K/\Q}(\mfn)|^\frac{1}{4[K: \Q]}$.
	We now state the main results of this section. 	
	
		\begin{thm}
			\label{main result2 abelian variety for  over K}
			Let $A/K$ be an abelian variety with conductor $\mfn$ and $A(K)[p] \neq 0$ for some $p \in \mathbb{P}$. Let $\mcL$ be a symmetric ample line bundle on $A$. Then there exist constants $C\in \R^+$ (depending on $A$, $K$, $p$) and $C^\prime \in \R^+$ (depending on $A$, $K$, $\mcL$, $p$) such that
			\begin{equation*}
				\label{Canonical bound for N_A(B) over K}
				\hat{N}_{A,K}(B) \leq B^{ \left(\frac{C}{{\log \log B}} \right)},\ N_{A,K}(B) \leq B^{ \left(\frac{C^\prime}{{\log \log B}} \right)}
			\end{equation*}
			holds for all $B \geq e^e$.
%			for all $B \geq \max\{e^e, H_0(A)\}$.
	\end{thm}

	However, for elliptic curves $E/K$, we will show that $C,\ C^\prime$
	depend only on $E, K, p$ and they are computable. Consider $E/K$ in the Weierstrass normal form
 \begin{equation}
 	\label{eqn for elliptic curve over K}
 	E: y^2=x^3+ax+b,
 \end{equation}
 where $a,b\in K$, with the minimal discriminant $\Delta_E$. Let $H_0(E):=|\mrm{N}_{K/ \Q}(\Delta_E)|^{\frac{1}{4[K : \Q]}}$.

\begin{thm}
	\label{main result2 for elliptic curve over K}
	Let $E/K$ be an elliptic curve as in \eqref{eqn for elliptic curve over K} with conductor $\mfn \subsetneq \mcO_K$ and $E(K)[p] \neq 0$ for some  $p \in \mathbb{P}$. Then there exist \textbf{computable} constants $C$, $C^\prime \in \R^+$ (depending on $E,\ K,\ p$) such that
	\begin{equation*}
		\label{Canonical bound for N_E(B) over K}
		\hat{N}_{E,K}(B) \leq B^{ \left(\frac{C}{{\log \log B}} \right)},  N_{E,K}(B) \leq B^{ \left(\frac{C^\prime}{{\log \log B}} \right)},
	\end{equation*}
    holds for all $B \geq \max\{e^e, H_0(E)\}$. Moreover, the constants $C, C^\prime$ depend only on $K,p$, if $\#E(K)[p]=p^2$.
\end{thm}

\begin{remark}
	Note that, in Theorem~\ref{main result2 for elliptic curve over K}, the assumption $\mfn \subsetneq \mcO_K$ is required to get at least one place of bad reduction for the elliptic curve $E/K$. This is needed in the proof of Lemma~\ref{bound for the rank of A/K terms of B} for the explicit upper bound of $\mrm{rank} (E(K))$ and in the proof of Theorem~\ref{main result2 for elliptic curve over K} to get the computable constants $C, C^\prime$. 
\end{remark}

We shall give proofs of Theorems~\ref{main result2 abelian variety for  over K},~\ref{main result2 for elliptic curve over K} in \S\ref{proof of main results for abelian and elliptic}. To prove these theorems, we need certain estimates, which are the content of the next three subsections.

\subsection{Bound on the rank of $A$ over $K$}
By Mordell-Weil theorem, the abelian group $A(K)$ is finitely generated.
%Since $A_\tors(K)$  is a finite abelian group and  {\color{red} generated by $2g$ elements},
%there exists integers $r_1, r_2,\dots,r_{2g}\geq1$ with $r_i | r_j$ for all $i <j$,  such that $A_{tors}(K) \simeq \Z / r_1\Z \oplus  \Z / r_2\Z \oplus \dots  \Z / r_{2g}\Z $.
% 	In , Ooe and Top gave an upper bound for the rank of $A(K)/mA(K)$, which will be useful for finding the rank of $A(K)$. More precisely;
	% 	\begin{prop}
% 		\label{OT1 upper bound for the rank of abelian variety}
% 		Let $A$ be an abelian variety over $K$ of dimension $g$, and let $m \geq2$ be an integer. Let $S:=\{v \in M_K^0 : A \text{ has bad reduction at } v\}\bigcup \{v \in M_K^0 : v(m) \neq 0\} \bigcup M_K^\infty $. If $A$ has full $m$-torsion over $K$ i.e., $A[m] \subseteq A(K)$, then $$\mrm{rank}_{\Z/m\Z} (A(K)/mA(K)) \leq 2g (\#S + \mrm{rank}_{\Z/m\Z}(\Cl_K[m])).$$
% 	\end{prop}
 For any ideal $\mfn \subseteq \mcO_K$, let $\omega_1(\mfn)$ denote the number of distinct prime ideals of $\mcO_K $ that occur in the prime factorization of $\mfn$. For any finite abelian group $G$, let $\varrho(G)$ denote the minimal number of generators of $G$.
 In the following lemma, we will find an upper bound of the rank of $A(K)$.
	%\begin{lem}
	%	\label{bound for the rank of abelian variety  over number field K}
	%	Let $A/K$ be an abelian variety of conductor $\mfn$, having full $m$-torsion for some $m >1$. Then $\mrm{rank} (A(K)) \leq { \color{red}2 g\omega_1(m)}+ 2g \omega_1(\mfn)+  2g \mrm{rank}_{\Z/m\Z} (\Cl_K[m])-2g$.
	%\end{lem}
	\begin{lem}
		\label{bound for the rank of abelian variety  over number field K}
		\label{bound for the rank of abelian variety having a non-trivial $m$-torsion}
		\label{bound for the rank of E having a non-trivial $m$-torsion}
		Let $A, \mfn, p$ be as in Theorem~\ref{main result2 abelian variety for  over K} and $L=K(A[p]).$ Then
		$$\mrm{rank} (A(K)) \leq 2g \left( 2[L: \Q] + [L: K] \omega_1(\mfn)+  \varrho (\Cl_L[p]) \right).$$
	\end{lem}
	
	\begin{proof}
		By Mordell-Weil theorem, $\mrm{rank} (A(L)) \leq \varrho (A(L)/pA(L)).$
		By the definition of $L$, $A$ has full $p$-torsion over $L$.
		Now applying~\cite[Theorem 2]{OT89} to $A/L$ with 
		$$S=\{v \in M_L^0 : A \text{ has bad reduction at } v\} \bigcup \{v \in M_L^0 : v(p) \neq 0\} \bigcup  M_L^\infty,$$
		 where  $M_L^\infty$ (resp.,  $M_L^0$) denote the archimedean (resp., non-archimedean) absolute values on $L$,
%		 Then  $\# \{v \in M_L^0 : A \text{ has bad reduction at } v\} \leq [L: K] \omega_1(\mfn)$, $\#  M_L^\infty \leq [L:\Q]$, and $\#\{v \in M_L^0 : v(p) \neq 0\} \leq [L:\Q]$. 
		the proof of the lemma follows since $\mrm{rank}(A(K)) \leq \mrm{rank}(A(L))$, $\# \{v \in M_L^0 : A \text{ has bad reduction at } v\} \leq [L: K] \omega_1(\mfn)$,  $\#\{v \in M_L^0 : v(p) \neq 0\} \leq [L: \Q]$, and $\#  M_L^\infty \leq [L: \Q]$.
	\end{proof}
	
	%The following lemma gives an upper bound for the rank of the elliptic curve $A/K$ having a non-trivial $m$-torsion point for some $m >1$.
	%%\section{abelian varieties}
	%
	%\begin{lem}
	%	\label{bound for the rank of abelian variety having a non-trivial $m$-torsion}
	%	Let $A/K$ be an abelian variety of conductor $\mfn$, having a non-trivial $m$-torsion point, for some $m >1$. Then $\mrm{rank} (A(K)) \leq g[L: K] \omega_1(m)+ g[L: K] \omega_1(\Delta_E)+  g \mrm{rank}_{\Z/m\Z} (\Cl_L[m])-2g$, where $L$ is the $m$-division field of the elliptic curve $A/K$.
	%\end{lem}
	%
	%\begin{proof}
	%	Let $L/K$ be the $m$-division field of the abelian variety $A/K$, i,e., the field generated by all the $m$-torsion points of $A/K$ and $K$. Then $A$ has full $m$-torsion over $L$. Now, applying Lemma~\ref{bound for the rank of abelian variety  over number field K} to the abelian variety $A$ over $L$, we have $\mrm{rank} (A(L)) \leq { \color{red}g[L: K] \omega_1(m)}+ g[L: K] \omega_1(\Delta_E)+  g \mrm{rank}_{\Z/m\Z} (\Cl_L[m])$.
	%\end{proof}	
	
	\subsection{Lower bound for $\hat{h}_K$ on $A/K$}
	For any abelian variety $A/K$, let $h_F(A)$ be the stable Falting height of $A$ (cf. \cite[\S1]{GR22}). In~\cite[Th\'eor\`eme 1.3]{GR22}, Gaudron and R\'emond gave an explicit lower bound for $\hat{h}_K$ on $A(K) \setminus A_{\tors}(K)$.
	 \begin{thm}
		\label{BS lower bound on the canonical height on A(K)}
		Let $A/K$ be an abelian variety of dimension $g$ and let $n =  [K: \Q] $. Then 
		$\hat{h}_K(P) \geq  {n \bigl( (6g)^8n \max\{1, h_F(A), \log n\} \bigr)^{-2g}}$ for all $ P \in A(K) \setminus A_{\tors}(K)$.
	\end{thm}
  However, for elliptic curves $E/K$, there is a conjecture of Lang (cf. \cite[page 92]{L78}) for explicit lower bound for $\hat{h}_K$ on $E(K) \setminus E_{\tors}(K)$.
%  , where the constant $c_1$ was more explicit.
	\begin{conj}[Lang]
		\label{Lang's conjecture for lower bound for height}
		Let $E/K$ be an elliptic curve with the  minimal discriminant $\Delta_{E}$. Then there exists a constant $c_K \in \R^+$ (depending on $K$) such that 
		\begin{equation*}
			\hat{h}_K (P) \geq c_K \log |\mrm{N}_{K/ \Q}(\Delta_{E})|
		\end{equation*}
		for all $P \in E(K) \setminus E_{\tors}(K)$.
	\end{conj}
%Let $\mathcal{F}_{E/K}$ denote the conductor of $E/K$. Define the Szpiro of $E/K$ by 
% $$\delta_{E/K}:=\frac{\log |N_{K/ \Q}(\Delta_{E})|}{\log |N_{K/ \Q}(\mathcal{F}_{E/K})|},$$
% when $E/K$ has atleaast one place of bad reduction. If $E/K$ has everywhere good reduction, then we define $\delta_{E/K}:=1$. By \cite{P06}, we have $\delta_{E/K} \geq 1$.
	In~\cite[Theorem 0.3]{HS88}, Hindry and Silverman proved a version of Conjecture~\ref{Lang's conjecture for lower bound for height}  with  $c_K$ depends exponentially on $[K: \Q]$ and Szpiro ratio $\sigma_{E/K}$ of $E$. In~\cite[Theorem 2]{P06}, Petsche proved that $c_K= \Big({10^{15} [K:\Q]^3 \sigma_{E/K}^6 \log^2 \left(104613 [K:\Q] \sigma_{E/K}^2 \right)} \Big)^{-1}$, which depends polynomially on $[K: \Q]$ and $\sigma_{E/K}$.

\subsection{Counting rational points through a covering argument in $\R^n$}
\label{covering argument of abelian variety over K}
By Mordell-Weil Theorem, $A(K) \simeq \Z^r \oplus A_{\tors}(K)$, where $r = \mrm{rank}(A(K))$. By taking the tensor product with $\R$, we get $A(K) \otimes \R \simeq \R^r$. Clearly, $A(K)/A_{\tors}(K)$ injects into $ A(K) \otimes \R$ via
$x \ra x \otimes 1$. In particular, $A(K)/A_{\tors}(K)$ sits inside $\R^r$ as a lattice $L_A$ of dimension $r$ via this injection (cf. \cite[page 3]{N21} for more details).

%Recall that, the N\'eron-Tate height $\hat{h}_K : A(K) \rightarrow [0, \infty)$ is a positive definite quadratic form on $A(K) \setminus A_{\tors}(K)$.
% (cf. \cite[Theorem B.5.3]{HS00}).
Recall that, the  N\'eron-Tate height $\hat{h}_K : A(K) \rightarrow [0, \infty)$ is a quadratic form, and positive definite on $A(K) / A_{\tors}(K)$. 
Define a norm on $\R^r$ by $|x|:= \sqrt{\hat{h}_K(P)}$ if $x=P$ for some $P \in A(K)/A_{\tors}(K)$ and this can be extended by continuity and linearly to $\R^r$ (cf. \cite[page 4]{BZ04} for more details). Since  $A(K)/ A_{\tors}(K)$ is a lattice $L_A \subseteq \R^r$ of dimension $r$, counting the points $P \in  A(K)/ A_{\tors}(K)$ with $\hat{h}_K(P) \leq \log B$ is same as counting the lattice points of $L_A$ in the ball of radius $\sqrt{\log B}$ centred at the origin in $\R^r$. To count the lattice points, we need the following covering lemma (cf. \cite[Lemma 4.1]{N21}).

% The following key steps are as follows:
% \begin{itemize}
% 	\item First we find a small enough radius $\rho_0$ such that any ball of radius of $\rho_0$ centred at a point of $L_A$ does not contain any other points of $L_A$.
% 	\item Next, we count total number of these balls to cover the intersection of $L_A$ with the ball of radius $\sqrt{\log B}$ centred at $0$.
% \end{itemize}

\begin{lem}
	\label{covering lemma}
	For a positive integer $n$, radii $R$, $\rho$ and a subset $S$ of $n$-ball $B_n(0,R)$, there exists a set of at most $(1+ \frac{2R}{\rho})^n$ balls of radius $\rho$ centred at points of $S$ such that $S$ is contained in the unions of these balls.
\end{lem}

\subsection{Proofs of~Theorems~\ref{main result2 abelian variety for  over K},~\ref{main result2 for elliptic curve over K}}
\label{proof of main results for abelian and elliptic}
%We are now in a position to give a proof of Theorem~\ref{main result2 abelian variety for  over K}.
% and its proof is similar to that of~\cite[Theorem 2.1]{N21}.
	\begin{proof}[Proof of Theorem~\ref{main result2 abelian variety for  over K}]
%		By Mordell-Weil Theorem for abelian varieties over $K$, the $\Z$-module $A(K)/A_{\tors}(K)$ is free of rank $r$.
		%By taking tensor product with $\R$, we get $	A(K) \otimes \R \simeq \R^r$ which is a lattice of dimension $r$ and say it as $L_A$. 
%		Similar to \S~\ref{section to prove thm2 for E/K}, $A(K)/A_{\tors}(K)$ sits inside the lattice $ \R^r$ via the injection.
		By Theorem~\ref{BS lower bound on the canonical height on A(K)} there exists $c_1 \in \R^+$ such that $\hat{h}_K (P) \geq  c_1$ for all $P \in A(K) \setminus A_{\tors}(K)$.
		Now applying Lemma~\ref{covering lemma} with $n=r$, $R= \sqrt{\log B}$ and $\rho=\sqrt{c_1}>0 $, we get
		$$ \hat{N}_{A,K}(B) \leq t \left( 1+ 2\sqrt{(\log B)/c_1}  \right)^r \leq \left(\log B \right)^{c_2+{c_3}r} \underset{(*)}{\leq} B^{\frac{c_2+c_3r}
		{\log \log B}} \underset{\mrm{Lemma}~\ref{bound for the rank of abelian variety  over number field K} }{\leq} B^{\frac{C}
		{\log \log B}} $$ 
	   for all $B \geq e^e$, where $c_2,c_3 \in \R^+$ (depending on $A,K$), $C\in \R^+$ (depending on $A,K, p$), and $t :=\# A_\tors(K)$.  In the inequality $(*)$, we used $\log B \leq B^{\frac{1} {\log \log B}}$ for $B \geq e^e$. This completes the proof of the first part.
		
		To prove the second part, we proceed as follows. By~\cite[Theorem B.5.1]{HS00}, there exists a constant $d\in \R^+$ (depending on $A,K, \mcL$) such that
		$|\hat{h}_K(P)- h_K(P)| \leq d $ for all $P \in A(K)$.
		Since $ h_K(P)= \log H_K(P)$, we get
		$$\{P \in  A(K): H_K(P) \leq B\} \subseteq \{P \in  A(K): \hat{h}_K(P) \leq \log B+d\}.$$
		This gives $N_{A,K}(B) \leq \hat{N}_{A,K}(\log B \ +d) \leq \hat{N}_{A,K}(e^ {\log B+d})=\hat{N}_{A,K}(B e^{d}) \leq (B e^{d})^ { \left(\frac{C}{{\log \log B e^{d}}} \right)}$.
		Here, the last inequality follows from the first part. Since $B \geq e^e$, we have 
		 $$N_{A,K}(B) \leq B^\frac{(d+1)C}{{\log \log B e^{d}}}  \leq B^\frac{C^\prime}{{\log \log B }},$$ where $C^\prime=(d+1)C$. We are done with the proof.
\end{proof}

\begin{remark}
	The constants $C \in \R^+$ in Theorem~\ref{main result2 abelian variety for  over K} can be computed in terms of $A, K, p$, as follows. 
%		\begin{enumerate}
			\item $C= \# A_\tors(K) + (1+ \frac{2}{\sqrt{c_1}}) \times 2g \left( 2[L: \Q] + [L: K] \omega_1(\mfn)+  \varrho (\Cl_L[p]) \right)$, where 
			$c_1= {n \bigl( (6g)^8n \max\{1, h_F(A), \log n\} \bigr)^{-2g}}$. 
				Here, $n=[K: \Q]$, $L=K(A[p])$, $ \varrho (\Cl_L[p])$ denote the minimal number of generators of the finite abelian group $\Cl_L[p]$, $\omega_1(\mfn)$ denote the number of distinct prime ideals of $\mcO_K $ that occur in the prime factorization of $\mfn$, and $h_F(A)$ be the stable Falting height of $A$.	
%		\end{enumerate} 
\end{remark}

We now prove Theorem~\ref{main result2 for elliptic curve over K}. For this,
we need the following lemma, which computes the upper bound for $\mrm{rank}(E(K))$ in terms of $B$. 
%For any $m \in \Z \setminus \{0\}$, let $\omega(m)$ denote the number of  $p \in \mathbb{P}$ dividing $m$.

\begin{lem}
	\label{lem for the bound of rank of abelian variety}
	\label{lem for the bound of rank of E}
	Let $E, \mfn, p$ be as in Theorem~\ref{main result2 for elliptic curve over K}. Then there exists a constant  $c_2 \in \R^+$ (depending on $E, K, p$) such that
	\begin{equation*}
		\label{bound for the rank of A/K terms of B}
		\mrm{rank} (E(K)) \leq  c_2\frac{\log B}{\log \log B},
	\end{equation*}
	holds for all $B \in \R^+$ with $B \geq \max \{e^e, H_0(E)\}$. Moreover, the constant $c_2$ depends only on $K, p$, if $\#E(K)[p]=p^2$.
\end{lem}
\begin{proof}
 Since $B \geq H_0(E):= |\mrm{N}_{K/ \Q}(\Delta_E)|^\frac{1}{4[K: \Q]}$, we get
	\begin{equation}
		\label{lower bound for B}
		%			\label{lower bound for B in terms of discriminant}
		\log |\mrm{N}_{K/ \Q}(\Delta_E)| \leq  4[K: \Q]\log B.	
	\end{equation}
	For any $m \in \Z \setminus \{0\}$, let $\omega(m)$ denote the number of distinct primes $p \in \mathbb{P}$ that divides $m$. Since $E/K$ has non-trivial conductor $\mfn$, we get $|\mrm{N}_{K/ \Q}(\Delta_E^2)| \geq 4$.
	By~\cite[Th$\acute{e}$or$\grave{e}$me 11]{R83},
    we get $\omega(|\mrm{N}_{K/ \Q}(\Delta_E)|)=\omega(|\mrm{N}_{K/ \Q}(\Delta_E^2)|) \leq  \frac{13841 \log|\mrm{N}_{K/ \Q}(\Delta_E^2)|}{\log\log |\mrm{N}_{K/ \Q}(\Delta_E^2)|}$. By~\eqref{lower bound for B},
	$\omega(|\mrm{N}_{K/ \Q}(\Delta_E)|) \leq c\frac{\log B}{\log \log B}$, where $c \in \R^+$ depends on $K$. Then
	 $$ \omega_1(\mfn) \leq \omega_1(\Delta_E) \leq [K : \Q] \times  \omega(|\mrm{N}_{K/ \Q}(\Delta_E)|) \leq c[K : \Q] \frac{\log B}{\log \log B}.$$
	Hence, the proof follows from Lemma~\ref{bound for the rank of abelian variety having a non-trivial $m$-torsion}.
\end{proof}

\begin{proof}[Proof of Theorem~\ref{main result2 for elliptic curve over K}]
  For any ideal $\mfa \subseteq \mcO_K$, let $\lambda(\mfa) :=  \log |\mrm{N}_{K/ \Q}(\mfa)|$.
	By~\cite[Theorem 2]{P06}, we get $\hat{h}_K (P) \geq \frac{\lambda(\Delta_E)}{c_3 \sigma_{E/K}^6 \log^2(c_4 \sigma_{E/K}^2)} $ for all $P \in E(K) \setminus E_{\tors}(K)$, where $c_3=10^{15}[K: \Q]^3$ and $c_4=104613[K: \Q]$. Since $\log^2(c_4 \sigma_{E/K}^2) \leq (c_4 \sigma_{E/K}^2)^2$, $\hat{h}_K (P) \geq \frac{\lambda(\Delta_E)}{c_5 \sigma_{E/K}^{10}}$, where $c_5 =c_3 c_4^2$. Since $E/ K$ has conductor $\mfn \subsetneq \mcO_K$, $\lambda(\Delta_E)>0$. By~\S\ref{covering argument of abelian variety over K},  $E(K)/E_{\tors}(K)$ sits inside $ \R^r$ as a lattice, where $r : =\mrm{rank}(E(K))$.
	
	Applying Lemma~\ref{covering lemma} with $n=r$, $R= \sqrt{\log B}$ and $\rho=\sqrt{\frac{\lambda(\Delta_E)}{c_5 \sigma_{E/K}^{10}}}>0  $, we get $ \hat{N}_{E,K}(B) \leq t \left(1+ 2\sqrt{\lambda(\Delta_E)} \frac{ \sqrt{c_5\log B}}{\lambda(\Delta_E)} \sigma_{E/K}^{5} \right)^r$, where
	$t := \# E_{\tors}(K)$. Since $\sigma_{E/K}=\frac{\lambda(\Delta_E)}{\lambda(\mfn)} \geq 1$ and by~\eqref{lower bound for B}, we have
%	$ \hat{N}_{E,K}(B) \leq t \left(  1+  4[K: \Q] \sqrt{c_5} \frac{\log(B)}{\lambda(\mfn)} \sigma_{E/K}^{4} \right)^r$. Since $\sigma_{E/K} \geq 1$, by~\eqref{lower bound for B} we get $\frac{4[K: \Q] \log B \sqrt{c_5}\sigma_{E/K}^{4}}{\lambda(\mfn) } \geq 1 $, hence
	\begin{equation}
		\label{bound for N_E inside the proof}
	\hat{N}_{E,K}(B) \leq t \left(  1+  4[K: \Q] \sqrt{c_5} \frac{\log(B)}{\lambda(\mfn)} \sigma_{E/K}^{4} \right)^r \leq t \left( c_6 \frac{ {\log B}}{\lambda(\mfn)} \sigma_{E/K}^{4}  \right)^r,
	\end{equation}
	where $c_6=8[K: \Q]\sqrt{c_5}$. Here the last inequality follows since $\frac{4[K: \Q] \log B \sqrt{c_5}\sigma_{E/K}^{4}}{\lambda(\mfn) } \geq 1 $. Now, we complete the proof of this theorem in two cases.
	
	We consider the first case $r <  \alpha_L$ with
	$\alpha_L:=30[L: \Q]+ 2 \varrho (\Cl_L[p])$, where $L = K(E[p])$.  Since the conductor
    $\mfn \subsetneq \mcO_K$, we get $\lambda(\mfn) \geq \log2$. By~\eqref{lower bound for B}, we have
		$\sigma_{E/K} \leq \frac{4[K: \Q] \log B}{\log 2}.$ Hence, by~\eqref{bound for N_E inside the proof}, we get
	$ \hat{N}_{E,K}(B) \leq c_7 \left(\log B  \right)^{c_8r}$ for some computable constants $c_7,\ c_8 \in \R^+$. Since $r <  \alpha_L$, we have
	$$ \hat{N}_{E,K}(B) \leq  \left( \log B  \right)^{C_1} \leq  \left( B  \right)^{\frac{C_1}{\log \log B}},$$
	for all $B \geq \mrm{max}\{e^e, H_0(E)\}$ and for a computable constant $C_1 \in \R^+$.
	
	We now consider the second case, i.e.,  when $\alpha_L \leq r$.
	By Lemma~\ref{lem for the bound of rank of E}, we have 
	\begin{equation}
		\label{range of r}
		\alpha_L \leq r \leq \alpha(B), 
		\end{equation} 
	where $\alpha(B) := c_2\frac{\log B}{\log \log B}.$	
	Since $\omega(\mrm{N}_{K/ \Q}(\mfn)) \geq \frac{\omega_1(\mfn)}{[K: \Q]}$, by Lemma~\ref{bound for the rank of E having a non-trivial $m$-torsion} we have
	\begin{equation}
		\label{bound for the number of prime divisor of norm of conductor}
		\omega(\mrm{N}_{K/ \Q}(\mfn))  \geq \frac{r-\beta_L}{2[L:\Q]} =:r_1,
	\end{equation}
    where $\beta_L=4[L: \Q] + 2 \varrho (\Cl_L[p])$. In particular, $|\mrm{N}_{K/ \Q}(\mfn)| \geq p(r_1)$, where $p(n)$ denote the product of the first $n$ rational primes. Since $r \geq \alpha_L$, we get $r_1 \geq 13.$ By~\cite[Lemma 4.5]{N21}, we have
    $p(r_1)\geq r_1^{r_1}$, hence $ \lambda(\mfn)= \log |\mrm{N}_{K/ \Q}(\mfn)| \geq r_1\log r_1$. By~\eqref{lower bound for B}, we get
		$\sigma_{E/K} \leq \frac{4[K: \Q] \log B}{r_1\log r_1}.$
    Using~\eqref{bound for N_E inside the proof}, we have $\hat{N}_{E, K}(B) \leq t \left(  \frac{ {c_9 \log B}}{ r_1\log r_1}  \right)^{5r}$ for a computable constant $c_9 \in \R^+$. Since $r=2[L: \Q]r_1+ \beta_L$, we get
	\begin{equation}
		\label{bound for N_E for large rank}
 \hat{N}_{E,K}(B) \leq t \left( \frac{c_9 {\log B}}{ r_1\log r_1}  \right)^{c_{10}r_1},
	\end{equation} 
	for a computable constant $c_{10}\in \R^+$. Since $r_1 \geq 13$ and $r_1 <r$, by ~\eqref{range of r} we get $13 \leq r_1 < \alpha(B)$. By~\eqref{bound for N_E for large rank}, it is enough to find the maximum of the function 
	$\phi (x)=  \left( \frac{ {\gamma}}{ x\log x}  \right)^{x}= e^{x(\log \gamma-\log(x\log x))}$ for $x\in [13, \alpha(B)],$ where $\gamma=c_9\log B$.  
	Then $$\phi' (x)= \phi (x) \times \left(\log\gamma-\log(x\log x)-1-\frac{1}{\log x} \right).$$
	Then, $\phi' (x) \geq 0$ for $x \leq x_0$ where $x_0 \log x_0= \gamma e^{-1 -\frac{1}{\log 13}}$,
	and $\phi' (x) \leq 0$ for $x \geq y_0$ where $y_0 \log y_0= \gamma e^{-1 -\frac{1}{\log \alpha(B)}}$. Therefore, the function $\phi (x)$ is increasing for $x \leq x_0$ and decreasing $x \geq y_0$.
	\begin{itemize}
	 \item For $13 \leq x \leq x_0$, $\phi$ has maximum at $x_0$ and
	$\phi(x) \leq e^{ \alpha(B)(1+ \frac{1}{\log 13}) }$.
	\item For $y_0\leq x \leq  \alpha(B)$, $\phi$ has maximum at $y_0$ and $\phi(x) \leq e^{ \alpha(B)(1+ \frac{1}{\log A}) } \leq e^{ \alpha(B)(1+ \frac{1}{\log 13}) }$. For $x_0 \leq x \leq y_0$,  
	$\phi(x) \leq e^{ \alpha(B)(1+ \frac{1}{\log 13}) }$.
	\end{itemize}
	Hence, $\phi(x) \leq e^{ (1+ \frac{1}{\log 13})  \alpha(B)} = B^{\frac{c_{11}}{\log \log B}}$ for all $x\in [13, \alpha(B)]$, where $c_{11}=c_2 (1+ \frac{1}{\log 13}).$ For any $t \geq 1, B \geq e^e$, we have $t \leq B^{\frac{t}{\log \log B}}$. Now, by~\eqref{bound for N_E for large rank}, we get
	\begin{equation}
		\label{N_E for the second part}
		\hat{N}_{E,K}(B) \leq  B^{\frac{C_2}{\log \log B}}
	\end{equation} 
	for a computable constant $C_2 \in \R^+$.  Finally, taking $C= \max\{ C_1, C_2 \}$, the first part of the theorem follows. From the proof, we see that the constants $C_1$ and $C_2$ are computable.

	Arguing as in the proof of Theorem~\ref{main result2 abelian variety for  over K}, the second part of the theorem follows from~\cite{Z76} and the estimate on $\hat{N}_{E,K}(B)$ from the first part.
%	To prove second part, we proceed as follows. By~\cite[Theorem XX]{Z76}, there exists constants $c_1$ and $c_2$ (depending on $E,K$) such that 
%	\begin{equation}
	%	\label{difference between heights}
	%	c_1 \leq \hat{h}_K(P)- h_K(P) \leq c_2
	%\end{equation} 
	%for any $P \in E(K)$.
    %the difference between the logarithmic and the canonical heights is bounded. So
 	%Hence $N_{E_K}(B) \leq \hat{N}_{E_K}(e^ {\log B+c_2})=\hat{N}_{E_K}(B e^{c_2}) \leq (B e^{c_2})^ { \left(\frac{D_K}{{\log \log B e^{c_2}}} \right)}$. Here, the first inequality follows from \eqref{difference between heights} and the last inequality follows from the first part. This gives $N_{E_K}(B) \leq B^\frac{(c_2+1)D_K}{{\log \log B e^{c_2}}}  \leq B^\frac{C_K}{{\log \log B }}$, where $C_K=(c_2+1)D_K$.
\end{proof}
% {\color{red}
% \begin{thm}[\cite{Z76}]
% 	Let $E/K$ be an elliptic curve as in \eqref{eqn for elliptic curve over K}. Then there exists absolutely computable constants $c_1$, $c_2$ such that
% 	$$c_1 \leq \hat{h}(P)-h(P) \leq c_2$$
% 	for all $P \in E(K)$.
% \end{thm}}

	\begin{remark}
		The constants $C$, $C^\prime \in \R^+$ in Theorem~\ref{main result2 for elliptic curve over K} can be computed explicitly, in terms of $E, K, p$, as follows. 
		\begin{enumerate}
			\item $C=\max\{C_1, C_2 \}$, where 
			\begin{enumerate}
				\item $C_1= \# E_\tors(K) + \left( 2^{11} \times 104613 \times (10 [K: \Q])^{\frac{15}{2}}
			(\log 2)^{-5}  +5 \right) \times  \Bigl( 30 [L: \Q] + 2 \varrho (\Cl_L[p]) \Bigr)$,
				
				\item 	$C_2= \# E_\tors(K) + c_{10} c_{11}$, where 
				     \begin{itemize}
				     	\item $c_{10}= 5  \Bigl( 6[L:\Q]+ 2\varrho (\Cl_L[p]) \Bigr)$,
				     	\item $c_{11}=2(1 +\frac{1}{\log 13}) \times  \Bigl(2[L:\Q] + 221456 [L:\Q] [K: \Q]+ \varrho (\Cl_L[p]) \Bigr) $. Here, $L=K(E[p])$ and $ \varrho (\Cl_L[p])$ denote the minimal number of generators of the finite abelian group $\Cl_L[p]$.
				     	
				     \end{itemize}
	 				
			\end{enumerate}

			\item $C^\prime= (d+1)C$, where $d=\frac{3\nu+7[K: \Q] \log 2}{2}$ with $\nu=-\sum_{v \in M_K} [K_v : \Q_v] \times \min \left\{ 0, \min\{ \frac{v(a)}{2}, \frac{v(b)}{3}\} \right\}$. Here $M_K$ denote the set of all absolute values on $K$.
		\end{enumerate} 
	\end{remark}

\subsection{Counting rational points on product of elliptic curves over $K$}
 Let $E_1,E_2, \ldots, E_g$ be elliptic curves over $K$, and let $A:= \prod_{i=1}^g E_i$ be the product of elliptic curves. Fix a Weierstrass embedding of each $E_i$ with $1 \leq i \leq g$, so that each $E_i$ can be considered as a subset of $\mathbb{P}^2$ through this embedding.
  Let $H_{A,K}$ be the multiplicative height associated with the resulting Segre embedding of $A$. Let $h_{A,K}$, $\hat{h}_{A,K}$ be the logarithmic,
  N\'eron-Tate heights associated with $H_{A,K}$.

  In this section, we extend Theorem~\ref{main result2 for elliptic curve over K} to the product of elliptic curves. To do this, we need to know the relation between the height of a point in $A(K)$ and the height of the corresponding points in $E_i(K)$.
\begin{lem}
	\label{relation between the logarthmic heights of A and E_i}
	\label{relation between the canonical heights of A and E_i}
	Let $E_1,E_2, \dots, E_g$ be elliptic curves over $K$, and $A:= \prod_{i=1}^g E_i$. Let $ P:=(P_1,P_2, \dots, P_g) \in A(K)$ with $P_i \in E_i(K)$ for all $i$.
	Then, we have
$H_{A, K} (P)= \prod_{i=1}^g H_K (P_i)$, $h_{A, K} (P) =	\sum_{i=1}^g h_K (P_i)$, and  $\hat{h}_{A, K} (P)= \sum_{i=1}^g \hat{h}_K (P_i),$
where $H_K (P_i),\ h_K (P_i),\  \hat{h}_K (P_i)$ are the multiplicative, logarithmic, N\'eron-Tate height of $P_i \in E_i(K)$ via the Weierstrass embedding of $E_i(K)$ in $\mathbb{P}^2(K)$.
\end{lem}
\begin{proof}
%  First, we will prove this lemma for $g=2$. Let $f: A(K)= E_1(K) \times E_2(K)  \rightarrow \mathbb{P}^8(K)$ be the Segre embedding defined by
%  $$	f\left( (x_0;x_1;x_2), (y_0;y_1;y_2)\right)=(x_iy_j; 0\leq i \leq 2, 0\leq j \leq 2).$$
%
% 	By the definition of $H_{K}$, we have $H_{K} (P)= H_K(f(P))$ for any $P \in A(K)$. Let $P =(P_1,P_2) \in A(K)$ for some $P_1 \in E_1(K)$ and $P_2 \in E_2(K)$. Let $P_1=(x_0;x_1;x_2) \in \mathbb{P}^2, P_2=(y_0;y_1;y_2) \in \mathbb{P}^2$ for  $x_i,y_j \in K$ with $0\leq i,j \leq 2$. Then $f(P)=(x_iy_j; 0\leq i \leq 2, 0\leq j \leq 2).$ By the definition of $H_{K}$, we have $H_{K} (P)= \prod_{v\in M_K} \max \{||x_iy_j||_v,  0\leq i,j \leq 2 \}$. Since $||.||_v$ is multiplicative for each $v \in M_K$, we have
% %	$$H_{K} (P)=\prod_{v\in M_K} \max \{||x_i||_v,  0\leq i\leq 2 \} \times \prod_{v\in M_K} \max \{||y_j||_v,  0\leq j \leq 2 \}.$$
% $H_{K} (P)=H_K (P_1) \times H_K (P_2)$, hence $h_{K} (P)=h_K (P_1)+h_K (P_2)$.

For any $g \geq 2$, let $\psi(g) :  A(K) : =  \prod_{i=1}^g E_i(K) \rightarrow \mathbb{P}^{3^g-1}(K)$ denote the Segre embedding. For any $P:=(P_1,P_2, \dots, P_g) \in A(K)$, let  $\psi(g-1)(P_1,P_2, \dots, P_{g-1})= (x_0, x_1, \dots x_{3^{(g-1)}-1})$ for some $x_i \in K$. Suppose $P_g=(y_0;y_1;y_2) \in \mathbb{P}^2$.  By definition, $H_{A,K} (P)= H_{A,K} (\psi(g)(P))= \prod_{v\in M_K} \max \{||x_iy_j||_v,  0\leq i \leq 3^{(g-1)}-1,\ 0\leq j \leq 2 \}$, where $M_K$ denote the set of all absolute values on $K$.
Since $||.||_v$ is multiplicative for each $v \in M_K$, we have
$$H_{A,K} (P)=H_{A^\prime,K} (\psi(g-1)(P_1,P_2, \dots, P_{g-1})) \times H_K (P_g)= H_{A^\prime,K} (P_1,P_2, \dots, P_{g-1}) \times H_K (P_g),$$
where $A^\prime :=\prod_{i=1}^{g-1} E_i$ and $H_{A^\prime,K}$ be the multiplicative height associated with the resulting Segre embedding of $A^\prime$. By induction
on $g$, we get $H_{A, K} (P)= \prod_{i=1}^g H_K (P_i)$.
By taking logarithms on both sides, we get  $h_{A,K} (P)= \sum_{i=1}^g h_K(P_i)$.
By the definition of $\hat{h}_{A, K}$, we have $\hat{h}_{A, K}(P)= \lim_{n \rightarrow \infty} \frac{h_{A, K}([2^n]P)}{4^n} =\lim_{n \rightarrow \infty} \frac{h_{A, K}([2^n]P_1, [2^n]P_2, \ldots, [2^n]P_g)}{4^n} \newline
= \lim_{n \rightarrow \infty} \frac{\sum_{i}^g h_{K}([2^n]P_i)}{4^n} = \sum_{i=1}^g \hat{h}_K (P_i).$ This completes the proof.
\end{proof}

By applying Theorem~\ref{main result2 for elliptic curve over K} to each $E_i$ and combining it with Lemma~\ref{relation between the logarthmic heights of A and E_i}, we get the following proposition.
\begin{prop}
		\label{main result1 for product of elliptic curves over K}
		Let $E_1,E_2, \dots E_g$ be elliptic curves over $K$ as in \eqref{eqn for elliptic curve over K} with conductor $\mfn_i\subsetneq \mcO_K$, and let $A:= \prod_{i=1}^g E_i$. Assume $E_i(K)[p] \neq 0$, for some $p \in \mathbb{P}$ and for all $i=1,\dots, g$. Then there exist \textbf{computable} constants $C$,\ $C^\prime$ (depending on $E_i\ (1 \leq i \leq g)$, $K,\ p,\ g$), such that
				\begin{equation*}
 				\hat{N}_{A,K}(B) \leq B^{ \left(\frac{C}{{\log \log B}} \right)},\ N_{A,K}(B) \leq B^{ \left(\frac{C^\prime}{{\log \log B}} \right)}
			\end{equation*}
		 for all $B \geq \max\{e^e, H_0(E_i) | 1 \leq i \leq g\}$.
	Moreover, the constants $C$, $C'$ depend only on $K,\ p,\ g$, if $\#E_i(K)[p]=p^2$ for all $i=1,\dots, g$.
\end{prop}

	Note that the constants $C$ $C^\prime$ in Theorem~\ref{main result2 abelian variety for  over K} are undetermined, whereas the constants in Proposition~\ref{main result1 for product of elliptic curves over K} are computable.

In the following proposition, we compute an explicit lower bound for the  N\'eron-Tate height on the product of elliptic curves in terms of $[K:\Q]$, discriminants, and Szpiro ratios $\sigma_{E_i/K}$ of $E_i/K$. In general, such a lower bound is not explicitly known for abelian varieties (cf. ~\cite[Theorem 0.1]{BS04}).

\begin{prop}
	\label{prop for lower bound for product of elliptic curves}
	Let $E_1, E_2, \dots, E_g$ be elliptic curves over $K$, and let $A:= \prod_{i=1}^g E_i$ be the product of $\{ E_i \}_{i=1}^g$. For any $ P \in A(K) \setminus A_\tors(K)$, we have
	$$\hat{h}_{A,K} (P) \geq \frac{1}{10^{15} [K:\Q]^3 C_{A,2}^6  \log^2(104613 [K:\Q]C_{A,2}^2)} C_{A,1}, $$
	where $C_{A,1}= \min\{\log |\mrm{N}_{K/ \Q}(\Delta_{E_i})|, i=1,\dots g\}$ and $C_{A,2}= \max\{\sigma_{E_1/K}, \dots \sigma_{E_g/K}\}$.
\end{prop}
\begin{proof} 
	Let $ P:=(P_1,P_2, \dots, P_g) \in A(K) \setminus A_\tors(K)$. By Lemma~\ref{relation between the canonical heights of A and E_i}, $P_i \in E_i(K) \setminus {(E_i)}_\tors(K)$ for some $i \in \{1, \dots g\}$.
Now, applying~\cite[Theorem 2]{P06} to $E_i$, we have 
$$\hat{h}_K (P_i) \geq \frac{1}{10^{15} [K:\Q]^3 \sigma_{E_i/K}^6 \log^2(104613 [K:\Q] \sigma_{E_i/K}^2)}  \log |\mrm{N}_{K/ \Q}(\Delta_{E_i})|.$$ Taking $C_{A,1}= \min\{\log |\mrm{N}_{K/ \Q}(\Delta_{E_i})|, i=1,2,\dots g\}$ and $C_{A,2}= \max\{\sigma_{E_1/K}, \dots \sigma_{E_g/K}\}$, the proof of the proposition follows from Lemma~\ref{relation between the canonical heights of A and E_i}.
\end{proof}

\section*{Acknowledgments}
The authors thank Prof. J.H. Silverman, Prof. D. W. Masser, Prof. A. Galateau for their suggestions and comments on an earlier draft. The authors also thank Prof. A. Plessis for the Remark~\ref{RemarkBogomolov} in the article. This research was supported in part by the International Centre for Theoretical Sciences (ICTS) for participating in the program - ICTS Rational Points on Modular Curves (code: ICTS/RPMC-2023/09).

\end{document}